\documentclass[a4paper,12pt]{amsart}
\usepackage{amssymb}
\usepackage{amsmath, amsthm, amscd, amsfonts, amssymb, graphicx, color}
\usepackage{amsmath, amsthm, amscd, amsfonts, amssymb, graphicx, color}

\newtheorem{theorem}{Theorem}[section]

\newtheorem{corollary}[theorem]{Corollary}
\theoremstyle{definition}
\newtheorem{definition}[theorem]{Definition}
\newtheorem{example}[theorem]{Example}

\theoremstyle{remark}
\newtheorem{remark}[theorem]{Remark}
\numberwithin{equation}{section}

\title[Linear dynamics of operators on $B_0(\mathcal{H})$]
{Linear dynamics of operators on $B_0(\mathcal{H})$}

\author[S. Ivkovi\'{c}]{Stefan Ivkovi\'{c}}
\address{The Mathematical Institute of the Serbian Academy of Sciences and Arts,
	p.p. 367, Kneza Mihaila 36, 11000 Beograd, Serbia}
\email{stefan.iv10@outlook.com }

\author[S.M. Tabatabaie]{Seyyed Mohammad Tabatabaie}
\address{Department of Mathematics, University of Qom, Qom, Iran}
\email{sm.tabatabaie@qom.ac.ir}



\date{\today}

\begin{document}

 \maketitle

\begin{abstract}
In this paper we characterize hypercyclic translation operators on $B_0(\mathcal{H})$, the space of all compact linear operators on a Hilbert space $\mathcal{H}$. Also, we give some sufficient condition for a related  cosine operator function to be chaotic or topologically transitive.
\end{abstract}

\baselineskip17pt
 
\section{Introduction}
Linear dynamic of bounded operators on Banach spaces has been studied during the last decades; see the classic books \cite{bmbook,gpbook} as monographs.
One of the important notions in this field is \emph{hypercyclicity} which arises from the invariant closed subset problem in Mathematical Analysis. This concept is related to some other notions such as topological transitivity, topological mixing and chaos. Especially, hypercyclicity of weighted translation operators on Lebesgue spaces in the context of locally compact groups and hypergroups have been studied; see \cite{ge00, chta2, cc11, chen11}. Recently, the focus was on studying hypercyclicity  of operators on other special function spaces such as Orlicz spaces and solid Banach function spaces; see \cite{ccot, cd18, chta3}.
In this paper, we study some linear dynamical properties of a translation operator on the space of compact linear operators on a given Hilbert space. These operators are also called elementary operators.
In Section 3, as a main result (Theorem \ref{thm32}), we will give some necessary and sufficient conditions for such operators to be   hypercyclic.  This theorem is proved under an orthogonality assumption which plays a key role in the proof. In Example \ref{ex33} we present a class of unitary operators satisfying this property. In this section also we give a sufficient condition for a translation operator to be chaotic. In Section 4, we study linear dynamical properties of related cosine operator functions, and present a sufficient condition for these functions to be chaotic or topologically transitive. See \cite{sha, kostic, kalmes10, chen15, chen141, chache} as studies on some other classes of cosine operator functions. In Section 5, we investigate the above properties for the adjoint of such operators. Moreover, in Example \ref{ex34}  we introduce some operators satisfying the sufficient conditions from our results. 

For convenience of the reader, we first recall some notations and definitions related to linear dynamical systems.
\section{Preliminaries}
If $\mathcal X$ is a Banach space, the set of all bounded linear operators from $\mathcal X$ into $\mathcal X$ is denoted by $B(\mathcal X)$. Also, we denote $\Bbb{N}_0:=\Bbb{N}\cup\{0\}$.
\begin{definition}
	Let $\mathcal X$ be a Banach space. A sequence $(T_n)_{n\in \Bbb{N}_0}$  of operators in $B(\mathcal X)$ is called {\it topologically transitive} if for each non-empty open subsets $U,V$ of
	${\mathcal X}$, $T_n(U)\cap V\neq \varnothing$ for some $n\in \Bbb{N}$. If $T_n(U)\cap V\neq \varnothing$ holds from some $n$ onwards, then
	$(T_n)_{n\in \Bbb{N}_0}$ is called {\it topologically mixing}. 
\end{definition}
\begin{definition}
	Let $\mathcal X$ be a Banach space. A sequence $(T_n)_{n\in \Bbb{N}_0}$  of operators in $B(\mathcal X)$ is called {\it hypercyclic} if there is an element $x\in\mathcal X$ (called \emph{hypercyclic vector}) such that the orbit $\mathcal O_x:=\{T_nx:\,n\in\mathbb N_0\}$ is dense in $\mathcal X$. The set of all hypercyclic vectors of a sequence $(T_n)_{n\in \Bbb{N}_0}$ is denoted by $HC((T_n)_{n\in \Bbb{N}_0})$. If $HC((T_n)_{n\in \Bbb{N}_0})$ is dense in $\mathcal X$, the sequence $(T_n)_{n\in \Bbb{N}_0}$ is called \emph{densely hypercyclic}. An operator $T\in B(\mathcal X)$ is called \emph{hypercyclic} if the sequence $(T^n)_{n\in \mathbb N_0}$ is hypercyclic.
\end{definition}
Note that a sequence $(T_n)_{n\in \Bbb{N}_0}$  of operators in $B(\mathcal X)$ is topologically transitive if and only if it is densely hypercyclic \cite{gpbook}. Also, a Banach space admits a hypercyclic operator if and only if it is separable and infinite-dimensional \cite{bg99}. So, in this paper we assume that Banach spaces are separable and infinite-dimensional.
\begin{definition}
	Let $\mathcal X$ be a Banach space, and $(T_n)_{n\in \Bbb{N}_0}$  be a sequence of operators in $B(\mathcal X)$. A vector $x\in \mathcal X$ is called a {\it periodic element} of $(T_n)_{n\in \Bbb{N}_0}$ if there exists a constant $N\in\mathbb N$ such that for each $k\in\mathbb N$, $T_{kN}x=x$. The set of all periodic elements of $(T_n)_{n\in \Bbb{N}_0}$ is denoted by
${\mathcal P}((T_n)_{n\in \Bbb{N}_0})$. The sequence $(T_n)_{n\in \Bbb{N}_0}$ is called {\it chaotic} if $(T_n)_{n\in \Bbb{N}_0}$ is topologically transitive and ${\mathcal P}((T_n)_{n\in \Bbb{N}_0})$ is dense in ${\mathcal X}$. An operator $T\in B(\mathcal{X})$ is called \emph{chaotic} if the sequence $\{T^n\}_{n\in \Bbb{N}_0}$ is chaotic. 
\end{definition}
\section{Linear Dynamics of Translation Operators}
In this paper, $\mathcal H$ is a Hilbert space. While $\mathcal{H}$ is assumed to be separable, we set $\{e_j\}_{j\in\mathbb{N}}$ as its orthonormal basis. In this case, we denote 
$L_k:=\text{span}\{e_1,\ldots,e_k\}$ for all $k\in\mathbb{N}$.
The set of all bounded linear operators from $\mathcal H$ to $\mathcal H$ is denoted by $B(\mathcal H)$. Also, the set of all compact (finite rank, respectively) elements of $B(\mathcal H)$ is denoted by $B_0(\mathcal H)$ ($B_{00}(\mathcal H)$, respectively). If $\mathcal{M}$ is a closed subspace of $\mathcal{H}$, we denote the orthogonal projection of $\mathcal{H}$ on $\mathcal{M}$ by $P_{\mathcal{M}}$. By \cite[Proposition 3.3. Chapter II]{conw}, $\|P_{\mathcal{M}}\|=1$. The image and kernel of each $F\in B(\mathcal{H})$ are denoted by ${\rm Im}(F)$ and ${\rm Ker}(F)$, respectively. For each $F\in B(\mathcal{H})$ we define 
$$m(F):=\inf_{\|x\|=1}\|F(x)\|=\sup\{C\geq 0:\,\|F(x)\|\geq C\,\|x\|\,\,\text{for all}\,x\in\mathcal{H}\}.$$
Easily, one can see that for each $n\in\mathbb{N}$, $m(F^n)\geq m(F)^n$.
\begin{definition}
	Let $U,W\in B(\mathcal H)$. We define the operator $T_{U,W}:B(\mathcal H)\rightarrow B(\mathcal H)$ by
\begin{equation}
T_{U,W}(F):=WFU
\end{equation}
for all $F\in B(\mathcal H)$.
\end{definition}
Trivially, $T_{U,W}(B_0(\mathcal{H}))\subseteq B_0(\mathcal{H})$. 
If $U$ and $W$ are invertible, then $T_{U,W}$ is invertible and we have
\begin{equation*}
T^{-1}_{U,W}(F)=W^{-1}FU^{-1},\quad(F\in B(\mathcal H)).
\end{equation*}
In this case, simply we put $S_{U,W}:=T^{-1}_{U,W}$.
.
\begin{theorem}\label{thm32}
Let $\mathcal{H}$ be a	separable Hilbert space. Let $W\in B(\mathcal H)$  be invertible and $U\in B(\mathcal H)$ be unitary such that for each $k\in\mathbb{N}$ there exists an $N_k\in\mathbb{N}$ with
\begin{equation}\label{eq000}
U^n(L_k)\perp L_k \quad\text{for all } n\geq N_k.
\end{equation}
 Then, the following statements are equivalent. 
 \begin{itemize}
	\item [(i)] $T_{U,W}$ is hypercyclic on $B_0(\mathcal{H})$, where $B_0(\mathcal{H})$ is equipped with the operator norm $\|\cdot\|$.
	\item [(ii)] For each $m\in\mathbb{N}$ there exist a strictly increasing sequence $\{n_k\}$ in $\mathbb{N}$ and the sequences $\{D_k\}$ and $\{G_k\}$ of operators in $B_0(\mathcal{H})$ such that 
		\begin{equation}\label{cond2}
	\lim_{k\rightarrow\infty}\|D_k-P_m\|=\lim_{k\rightarrow\infty}\|G_k-P_m\|=0,
	\end{equation}
	and
	\begin{equation}\label{cond22}
	\lim_{k\rightarrow\infty}\left\|W^{n_k}G_k \right\|=\lim_{k\rightarrow\infty}\left\|W^{-n_k}D_k \right\|=0,
	\end{equation}
	where $P_m$ denotes the orthogonal projection onto $L_m$.
\end{itemize}
\end{theorem}
\begin{proof}
$\text{(i)}\Rightarrow\text{(ii)}$:  Let $T_{U,W}$ be hypercyclic on $B_0(\mathcal{H})$. Let $m\in\mathbb{N}$. Since $P_m$ belongs to $B_0(\mathcal{H})$, for each $k\in\mathbb{N}$ one can find an operator $F_k\in B_0(\mathcal{H})$ and a number $n_k\in\mathbb{N}$ such that 
\begin{equation}
\|F_k-P_m\|\leq\frac{1}{4^k},\qquad \|T^{n_k}_{U,W}(F_k)-P_m\|\leq \frac{1}{4^k}.
\end{equation}
Clearly, we can assume that $N_m<n_1<n_2<\ldots$. Therefore, $P_mU^{n_k}P_m=0$ and $P_mU^{-n_k}P_m=0$ for all $k\in\mathbb{N}$. Hence, we get 
\begin{align*}
\frac{1}{4^k}&\geq \|F_k-P_m\|\geq \|F_k-P_m\|\,\|U^{n_k}\,P_m\|\\
&\geq \|(F_k-P_m)\,U^{n_k}P_m\|=\|F_kU^{n_k}P_m\|\\
&=\|W^{-n_k}(W^{n_k}F_kU^{n_k}P_m)\|.
\end{align*}
Moreover,
\begin{align*}
\frac{1}{4^k}&\geq \|T^{n_k}_{U,W}(F_k)-P_m\|=\|W^{n_k}F_k U^{n_k}-P_m\|\\
&\geq \|W^{n_k}F_k U^{n_k}-P_m\|\,\|P_m\|\geq \|(W^{n_k}F_k U^{n_k}-P_m)P_m\|\\
&=\|W^{n_k}F_k U^{n_k}P_m-P_m\|.
\end{align*}
Further, 
\begin{align*}
\frac{1}{4^k}&\geq \|W^{n_k}F_k U^{n_k}-P_m\|\geq \|W^{n_k}F_k U^{n_k}-P_m\|\, \|U^{-n_k}P_m\|\\
&\geq \|(W^{n_k}F_k U^{n_k}-P_m)U^{-n_k}P_m\|=\|W^{n_k}F_k P_m\|,
\end{align*}
and
\begin{align*}
\frac{1}{4^k}&\geq \|F_k-P_m\|\geq \|F_k-P_m\|\,\|P_m\|\\
&\geq \|(F_k-P_m)\,P_m\|=\|F_kP_m-P_m\|.
\end{align*}
Set for each $k$, 
$$G_k:=F_kP_m,\qquad D_k:=W^{n_k}F_kU^{n_k}P_m.$$
This completes the proof.

$\text{(ii)}\Rightarrow \text{(i)}$: Assume that the statement $\text{(ii)}$
 holds. Equivalently, we show that $T_{U,W}$ is topologically transitive on $B_0(\mathcal{H})$.  Let $\mathcal{O}_1$ and $\mathcal{O}_2$ be two non-empty open subsets of $B_0(\mathcal{H})$. Since $B_{00}(\mathcal{H})$ is dense in $B_0(\mathcal{H})$, we can pick
$$F\in \mathcal{O}_1\cap B_{00}(\mathcal{H})\quad\text{and}\quad G\in\mathcal{O}_2\cap B_{00}(\mathcal{H}).$$
Choose an $\varepsilon>0$ such that the $\varepsilon$-balls around $F$ and $G$ are contained in $\mathcal{O}_1$ and $\mathcal{O}_2$, respectively. Since $F$ and $G$ are finite rank operators, there exists an $m_\varepsilon$ in $\mathbb{N}$ such that
$$\|P_{m_\varepsilon}F-F\|\leq \frac{\varepsilon}{2},\qquad \|P_{m_\varepsilon}G-G\|\leq \frac{\varepsilon}{2}.$$
By the assumption $\text{(ii)}$, for each $m:=m_\varepsilon$ there exist a strictly increasing sequence $\{n_k\}\subseteq\mathbb{N}$ and the sequences $\{D_k\}$ and $\{G_k\}$ of operators in $B_0(\mathcal{H})$ satisfying the relations \eqref{cond2} and \eqref{cond22}. For each $k\in\mathbb{N}$ put 
$$\Phi_k:=G_kF+S^{n_k}_{U,W}(D_kG).$$
We obtain
\begin{align*}
\|\Phi_k-F\|&\leq \|\Phi_k-P_{m_\varepsilon}F\|+\|P_{m_\varepsilon}F-F\|\\
&\leq \|G_k-P_{m_\varepsilon}\|\,\|F\|+\|W^{-n_k}D_k\|\,\|G\|+\frac{\varepsilon}{2}\rightarrow \frac{\varepsilon}{2},
\end{align*}
as $k\rightarrow\infty$. Further, 
\begin{align*}
\|T^{n_k}_{U,W}(\Phi_k)-G\|&\leq \|T^{n_k}_{U,W}(\Phi_k)-P_{m_\varepsilon}G\|+\|P_{m_\varepsilon}G-G\|\\
&\leq \|W^{n_k}G_kFU^{n_k}+D_kG -P_{m_\varepsilon}G\|+\frac{\varepsilon}{2}\\
&\leq \|W^{n_k}G_k\|\,\|F\|+\|D_k-P_{m_\varepsilon}\|\,\|G\|+\frac{\varepsilon}{2}\\
&\rightarrow \frac{\varepsilon}{2},
\end{align*}
as $k\rightarrow\infty$.
 This implies that for a large enough $k\in\mathbb{N}$ we have
$$T^{n_k}_{U,W}(\mathcal{O}_1)\cap \mathcal{O}_2\neq \varnothing,$$
and therefore, $T_{U,W}$ is topologically transitive.
\end{proof}
Now, we give an example of a unitary operator $U$ satisfying the assumption \eqref{eq000} in Theorem \ref{thm32}. Note that we have  \emph{not} used this assumption on $U$ for the proof of the implication ${\rm (ii)}\Rightarrow{\rm (i)}$ in Theorem \ref{thm32}.
\begin{example}\label{ex33}
	Let $X$ be a topological space. Let $\alpha:X\longrightarrow X$ be invertible, and $\alpha,\alpha^{-1}$ be Borel measurable. We say that $\alpha$ is \emph{aperiodic} if for each compact subset $K$ of $X$, there exists a constant $N>0$ such that for each $n\geq N$, we have $K \cap \alpha^{n}(K)=\varnothing$, where $\alpha^{n}$ means the $n$-fold combination of $\alpha$.
	
	Assume that $\alpha$ is an aperiodic mapping on $\mathbb{N}$ equipped with the discrete topology. So, for each $n\in\mathbb{N}$ there is a number $N_n\in\mathbb{N}$ such that 
	$$\alpha^m(\{1,2,\ldots,n\})\cap\{1,2,\ldots,n\}=\varnothing,$$
	for all $m\geq N_n$. For a separable Hilbert space $\mathcal{H}$, let $\{e_j\}_{j\in\mathbb{N}}$ denote an orthonormal basis of $\mathcal{H}$, and set $U_\alpha(e_j):=e_{\alpha(j)}$ for all $j\in\mathbb{N}$. Then, $U_\alpha$ is a unitary operator on $\mathcal{H}$ satisfying  
	$$U_\alpha^m(L_n)\perp L_n$$
	for each $n\in\mathbb{N}$ and all $m\geq N_n$, where $L_n:=\text{span}\{e_1,\ldots,e_n\}$.
\end{example}
\begin{example}\label{ex34}
	Let $\{e_j\}_{j\in\mathbb{N}}$ be an orthonormal basis for a Hilbert space $\mathcal{H}$. Define $W\in B(\mathcal{H})$ by
	\begin{eqnarray*}
		W(e_j):=
		\left\{
		\begin{array}{ll}
			\frac{1}{2}\,e_{j+2}, & \text{if } j \text{ is odd}, \\ \\
			2\,e_{j-2}, & \text{if } j \text{ is even and } j>2,\\ \\
			e_1, & \text{if } j=2.
		\end{array}
		\right.
	\end{eqnarray*}
Then, $W$ is invertible and $\|W\|=2$. For each fixed $k\in\mathbb{N}$ it is easily checked that $\|W^{2k-1+m}P_{2k}\|=\frac{1}{2^m}$ for all $m\in \mathbb{N}$. Consequently, $\|W^{2k-1+m}P_{2k-1}\|\leq\frac{1}{2^m}$. Further, it is also easily verified that for each $k,m\in\mathbb{N}$ we have
$\|W^{-2k-m}P_{2k+1}\|=\frac{1}{2^{m-1}}$, and this gives that $\|W^{-2k-m}P_{2k}\|\leq\frac{1}{2^{m-1}}$. As above, $P_n$ denotes the orthogonal projection onto $\text{span}\{e_1,\ldots,e_n\}$.

It follows that 
$$\|P_{2k}(W^*)^{2k-1+m}\|=\frac{1}{2^m},\qquad\|P_{2k+1}(W^*)^{-2k-m}\|=\frac{1}{2^{m-1}},$$
for all $k,m\in\mathbb{N}$. The operator $W$ satisfies the conditions of Theorem \ref{thm32}, Theorem \ref{thm38}, Theorem \ref{thm41}, Theorem \ref{thm47} and Theorem \ref{thm51}, whereas $W^*$ satisfies the conditions of Theorem \ref{thm54} and Theorem \ref{thm57}
\end{example}
\begin{definition}
	Let $\mathcal{X}$ be a Banach space, $a\in\mathcal{X}$, and $T\in B(\mathcal{X})$. We say that $T$ is \emph{$a$-transitive} if for each two non-empty open subsets $\mathcal{O}_1$ and $\mathcal{O}_2$ of $\mathcal{X}$ with $a\in \mathcal{O}_1$, there are $m,n\in\mathbb{N}$ such that 
	$$T^n(\mathcal{O}_1)\cap\mathcal{O}_2\neq\varnothing, \qquad T^m(\mathcal{O}_2)\cap\mathcal{O}_1\neq\varnothing.$$
\end{definition}
\begin{theorem}
	Let $U,W\in B(\mathcal{H})$ such that $W$ is invertible and $U$ is unitary. Then, the following statements are equivalent.
	\begin{itemize}
		\item [(i)] $T_{U,W}$ and $S_{U,W}$ are $0$-transitive on $B_0(\mathcal{H})$.
		\item [(ii)] For every finite dimensional subspace $K$ of $\mathcal{H}$ there are strictly increasing sequences $\{n_j\}$ and $\{m_j\}$ in $\mathbb{N}$ and sequences of operators 
		$\{G_j\}$ and $\{D_j\}$ in $B_0(\mathcal{H})$ such that
		\begin{equation}\label{cond3}
		\lim_{j\rightarrow\infty}\|G_j-P_K\|=\lim_{j\rightarrow\infty}\|D_j-P_K\|=0,
		\end{equation}
		and 
		\begin{equation}\label{cond32}
		\lim_{j\rightarrow\infty}\|W^{-m_j}G_j\|=\lim_{j\rightarrow\infty}\|W^{n_j}D_j\|=0.
		\end{equation}
	\end{itemize}
\end{theorem}
\begin{proof}
	${\rm (i)}\Rightarrow{\rm (ii)}$: Let $T_{U,W}$ and $S_{U,W}$ be $0$-transitive on $B_0(\mathcal{H})$, and let $K$ be a finite dimensional subspace of $\mathcal{H}$. Then, for each open subsets $\mathcal{O}_1$ and $\mathcal{O}_2$ of $B_0(\mathcal{H})$ with $0\in\mathcal{O}_1$ and $P_K\in\mathcal{O}_2$, there are $m,n\in\mathbb{N}$ such that 
	$$T^n_{U,W}(\mathcal{O}_2)\cap\mathcal{O}_1\neq\varnothing,\qquad S^m_{U,W}(\mathcal{O}_2)\cap\mathcal{O}_1\neq\varnothing.$$
	Using the relations $\|T^n_{U,W}(D)\|=\|W^nD\|$ and  $\|S^m_{U,W}(D)\|=\|W^{-m}D\|$ for all $D\in B_0(\mathcal{H})$, it is straightforward to check that ${\rm (ii)}$ holds.
	
		${\rm (ii)}\Rightarrow{\rm (i)}$: Let $\mathcal{O}_1$ and $\mathcal{O}_2$ be non-empty open subsets of $B_0(\mathcal{H})$ with $0\in\mathcal{O}_1$. Since $B_{00}(\mathcal{H})$ is dense in $B_0(\mathcal{H})$, there are $F\in B_{00}(\mathcal{H})\cap \mathcal{O}_2$ and $\varepsilon>0$ such that whenever $\|D-F\|<\varepsilon$, then $D\in\mathcal{O}_2$. Let $K$ be a finite dimensional subspace of $\mathcal{H}$ satisfying $P_KF=F$. With the corresponding conditions in ${\rm (ii)}$, one can choose $j\in\mathbb{N}$ large enough such that 
		$$\|G_j-P_K\|<\frac{\varepsilon}{\|F\|+1},\qquad \|W^{-m_j}G_j\|<\frac{\delta}{\|F\|+1},$$
		where $\delta>0$ is such that the open ball with center at 0 and radius $\delta$ is included in $\mathcal{O}_1$. Then,
		\begin{align*}
		\|S^{n_j}_{U,W}(G_jF)\|&\leq \|W^{-m_j}G_jF\|\\
		&\leq \|W^{-m_j}G_j\|\,\|F\|<\delta,
		\end{align*}
		so $S^{m_j}_{U,W}(G_jF)\in\mathcal{O}_1$. Moreover, 
		\begin{align*}
		\|G_jF-F\|&=\|(G_j-P_K)F \|\\
		&\leq \|G_j-P_K\|\,\|F\|<\varepsilon,
		\end{align*}
		so $G_jF\in \mathcal{O}_2$. Thus, $S^{m_j}_{U,W}(\mathcal{O}_2)\cap\mathcal{O}_1\neq\varnothing$. Similarly, one can find $j\in\mathbb{N}$ large enough such that 
		$T^{n_j}_{U,W}(\mathcal{O}_2)\cap\mathcal{O}_1\neq\varnothing$, and so the proof is complete.
\end{proof}
\begin{theorem}
	Let $U,W\in B(\mathcal H)$  such that $W$ be invertible and $U$ be unitary. If $T_{U,W}$ is hypercyclic on $B_0(\mathcal{H})$, then $m(W)<1<\|W\|$.
\end{theorem}
\begin{proof}
	 Suppose that $T_{U,W}$ is hypercyclic on $B_0(\mathcal{H})$, and so 0-transitive. Let $K$ be a  finite dimensional subspace of $\mathcal H$. Then, there are sequences $\{n_j\}$, $\{m_j\}$, $\{G_j\}$ and $\{D_j\}$ satisfying condition ${\rm (ii)}$ of the previous theorem. Since $\|P_K\|=1$, we have 
	 $$\lim_{j\rightarrow\infty}\|G_j\|=\lim_{j\rightarrow\infty}\|D_j\|=1.$$
	 Moreover, 
	 $$m(W)^{n_j}\,\|D_j\|\leq \|W^{n_j}D_j\|\rightarrow 0,$$
	 as $j\rightarrow\infty$, and 
	 $$m(W^{-1})^{m_j}\,\|G_j\|\leq \|W^{-m_j}G_j\|\rightarrow 0,$$
	 as $j\rightarrow\infty$. These imply that $m(W)<1<\|W\|$ since $\|W\|=m(W^{-1})^{-1}$.
\end{proof}
\begin{theorem}\label{thm38}
	Let $U,W\in B(\mathcal H)$  such that $W$ be invertible and $U$ be unitary. Suppose that there is a  finite dimensional subspace $K$ of $\mathcal{H}$ such that for a  constant $N>0$, $U^{n}(K)\perp K$ for all $n\geq N$.
	Then, we have 
	${\rm (i)}\Rightarrow{\rm (ii)}$:
	\begin{itemize}
		\item [(i)] $P_K$ belongs to the closure of $\mathcal{P}(\{S^n_{U,W}\}_{n\in\mathbb{N}_0})$ in $B_0(\mathcal{H})$. 
		\item [(ii)] There exists an increasing sequence $(n_k)$ in $\mathbb{N}$ such that $m(W^{-n_k})\rightarrow 0$ as $k\rightarrow\infty$.
	\end{itemize}
\end{theorem}
\begin{proof}	
	Let $K$ be a finite dimensional subspace of $\mathcal{H}$ such that $U^n(K)\perp K$ for all $n\geq N$. For each $k\in\mathbb{N}$ there exists an operator $F_k\in B_0(\mathcal{H})$ and a natural number $n_k$ such that
	$$\frac{1}{k^2}\geq \|F_K-P_K\|\quad\text{and}\quad F_k=S_{U,W}^{n_k}(F_k).$$
	We can suppose that $N\leq n_1<n_2<\ldots$. Hence, for each $k\in\mathbb{N}$,
	$$\frac{1}{k^2}\geq \|(F_k-P_K)\,P_K\|=\|F_kP_K-P_K\|.$$
	This gives $\|F_kP_K\|\geq 1-\frac{1}{k^2}$. So, for each $k\in\mathbb{N}$, there exists an $x_k\in\mathcal{H}$ with $x_k\neq 0$ and 
	$$\|F_kP_Kx_k\|\geq (1-\frac{1}{k^2})\,\|x_k\|.$$
	Next, we have 
	\begin{align*}
	\frac{1}{k^2}&\geq \|F_k-P_K\|=\|S_{U,W}^{n_k}(F_k)-P_K\|\\
	&=\|W^{-n_k}F_kU^{-n_k}-P_K\|\\
	&\geq\|(W^{-n_k}F_kU^{-n_k}-P_K)P_{U^{n_k}(K)}\|\\
	&=\|(W^{-n_k}F_kU^{-n_k} U^{n_k}P_K U^{-n_k}\|\\
	&=\|(W^{-n_k}F_kP_K\|.
	\end{align*}
	Hence, for each $k$ we get
	\begin{align*}
	\frac{1}{k^2}\,\|x_k\|&\geq \|(W^{-n_k}F_kP_K\|\,\|x_k\|\\
	&\geq \|(W^{-n_k}F_kP_K(x_k)\|\\
	&\geq (m(W^{-1}))^{n_k}\|F_kP_K x_k\|\\
	&\geq (m(W^{-1}))^{n_k}\,(1-\frac{1}{k^2})\,\|x_k\|.
	\end{align*}
	Dividing on the both sides of this inequality by 
	$\|x_k\|$, we get 
	$$\frac{1}{k^2}\geq (1-\frac{1}{k^2})\,(m(W^{-1}))^{n_k}$$
	for all $k\in\mathbb{N}$. 
\end{proof}
Recalling that $m(W^{-1})=\|W\|^{-1}$, we obtain the following result.
\begin{corollary}
	Let $U,W\in B(\mathcal H)$  such that $W$ be invertible and $U$ be unitary. Suppose that there exists a  finite dimensional subspace $K$ of $\mathcal{H}$ such that for a  constant $N>0$, $U^{n}(K)\perp K$ for all $n\geq N$ and  $\mathcal{P}(\{S^n_{U,W}\}_{n\in\mathbb{N}_0})$ is dense in $B_0(\mathcal{H})$, then 
		$\|W\|>1$.
\end{corollary}
Since $T_{U,W}=S_{U^{-1},W^{-1}}$, we can conclude the following fact. Just note that $m(W)=\|W^{-1}\|^{-1}$.
\begin{corollary}
	Let $U,W\in B(\mathcal H)$  such that $W$ be invertible and $U$ be unitary. Suppose that there is a  finite dimensional subspace $K$ of $\mathcal{H}$ such that for a  constant $N>0$, $U^{n}(K)\perp K$ for all $n\geq N$.
	If $\mathcal{P}(\{T^n_{U,W}\}_{n\in\mathbb{N}_0})$ is dense in $B_0(\mathcal{H})$, then  $m(W)<1$.
\end{corollary}
\begin{remark}
	If $W$ and $U$ are both unitary and there exists an $n\in\mathbb{N}$ with $W^n =U^n =I$, then $T^n = S^n =I$ and so obviously the set of all periodic elements of $\{T^n_{U,W}\}_{n\in\mathbb{N}_0}$ is dense in $B_0(\mathcal{H})$,	
	however $m(W) = m(W^{-1}) =1$. This shows that the above implication does not hold in general, and indeed we need the orthogonality assumption on $U$ in order to obtain that $m(W)$ is strictly less than 1.
\end{remark}
\begin{theorem}\label{thm310}
Let $\mathcal{H}$ be a separable Hilbert space and $U,W\in B(\mathcal H)$  such that $W$ be invertible and $U$ be unitary. Then, we have 
	${\rm (ii)}\Rightarrow{\rm (i)}$:
	\begin{itemize}
		\item [(i)] the operators $T_{U,W}$ and $S_{U,W}$ are chaotic on $B_0(\mathcal{H})$.
		\item [(ii)] For each $m\in\mathbb{N}$ there is a strictly increasing sequence $\{n_k\}\subseteq \mathbb{N}$ such that 
		$$\lim_{k\rightarrow\infty}\sum_{l=1}^\infty\|W^{ln_k}P_m\|=\lim_{k\rightarrow\infty}\sum_{l=1}^\infty\|W^{-ln_k}P_m\|=0,$$
		where the corresponding series are convergent for each $k$.
	\end{itemize}
\end{theorem}
\begin{proof}	
	By Theorem \ref{thm32}, (ii) implies that $T_{U,W}$ and $S_{U,W}$  are topologically transitive. So, it suffices to show that $\mathcal{P}(T^n_{U,W})$ and $\mathcal{P}(S^n_{U,W})$ are dense in $B_0(\mathcal{H})$. Let $\mathcal{O}$ be a non-empty open subset of $B_0(\mathcal{H})$. Then, there exists a finite rank operator $F\in\mathcal{O}$. Since $\|P_mF-F\|\rightarrow 0$ as $m\rightarrow\infty$, there exists an $m_0\in\mathbb{N}$ such that $P_{m_0}F\in\mathcal{O}$. Set $K:=L_{m_0}$. We may in the rest of proof assume that $F=P_KF$. 
	Choose a corresponding sequence $\{n_k\}\subseteq \mathbb{N}$ satisfying the assumption (ii). Observe also that for each $k,l\in\mathbb{N}$,
	\begin{align*}
	\|T^{ln_k}_{U,W}F\|&=\|W^{ln_k}F U^{ln_k}\|=\|W^{ln_k}F\|\\
	&=\|W^{ln_k}P_KF\|\leq \|W^{ln_k}P_K\|\,\,\|F\|.
	\end{align*}
	Similarly,
	\begin{equation*}
	\|S^{ln_k}_{U,W}F\|\leq \|W^{-ln_k}P_K\|\,\,\|F\|.
	\end{equation*}
	Set $$G_k:=\sum_{l=0}^\infty T^{ln_k}_{U,W}(F)+\sum_{l=1}^\infty S^{ln_k}_{U,W}(F).$$
	Easily, we have 
	$$T^{ln_k}_{U,W}(G_k)=G_k=S^{ln_k}_{U,W}(G_k)$$
	for all $l,k\in\mathbb{N}$, and $\lim_{k\rightarrow\infty}G_k=F$ in $B_0(\mathcal{H})$ since
	\begin{align*}
		\|G_k-F\|&\leq \sum_{l=1}^\infty \|T^{ln_k}_{U,W}(F)\|+\sum_{l=1}^\infty \|S^{ln_k}_{U,W}(F)\|\\
		&\leq\|F\|\, \sum_{l=1}^\infty\|W^{ln_k}P_K\|+\|F\|\,\sum_{l=1}^\infty\|W^{-ln_k}P_K\|.
	\end{align*}
	 This completes the proof. 
\end{proof}
\section{Cosine Operator Functions}
In this section, we intend to verify some dynamical properties of the cosine operator functions  related to the operator $T_{U,W}$.
If $U$ and $W$ are invertible, for each $n\in\mathbb{N}_0$ we put
\begin{equation}
C^{(n)}_{U,W}:=\frac{1}{2}\left(T^{n}_{U,W}+S^{n}_{U,W}\right).
\end{equation}
\begin{theorem}\label{thm41}
		Suppose that $U,W\in B(\mathcal{H})$ such that $W$ is invertible and $U$ is unitary. Then, we have ${\rm (ii)}\Rightarrow{\rm (i)}$:
	\begin{itemize}
	\item [(i)] The sequence $(C^{(n)}_{U,W})_{n\in\mathbb{N}_0}$ is topologically transitive on $B_0(\mathcal{H})$.
	\item [(ii)] For each $m\in\mathbb{N}$, there are sequences $(E_k)$ and $(R_k)$ of subspaces of $L_m$ and an strictly increasing sequence $(n_k)$ of positive integers such that $L_m=E_k\oplus R_k$ and
	\begin{equation}
	\lim_{k\rightarrow\infty}\left\|W^{n_k}P_m\right\|=\lim_{k\rightarrow\infty}\left\|W^{-n_k}P_m\right\|=0,
	\end{equation}
		\begin{equation}
	\lim_{k\rightarrow\infty}\left\|W^{2n_k}P_{E_k}\right\|=\lim_{k\rightarrow\infty}\left\|W^{-2n_k}P_{R_k}\right\|=0.
	\end{equation}
	\end{itemize}
\end{theorem}
\begin{proof}
	Let $\mathcal{O}_1$ and $\mathcal{O}_2$ be two non-empty open subsets of $B_0(\mathcal{H})$. Since $B_{00}(\mathcal{H})$ is dense in $B_0(\mathcal{H})$, we can pick
	$$F\in \mathcal{O}_1\cap B_{00}(\mathcal{H})\quad\text{and}\quad G\in\mathcal{O}_2\cap B_{00}(\mathcal{H}).$$
	Again, there exists an $m\in\mathbb{N}$ such that $P_mF\in\mathcal{O}_1$ and $P_mG\in\mathcal{O}_2$. Set $K:=L_m$.  So, there are corresponding sequences  $(E_k)$, $(R_k)$  and $(n_k)$ satisfying the condition ${\rm (ii)}$. Hence,
	\begin{equation*}
	\left\|T^{n_k}_{U,W}(P_K F)\right\|=\left\|W^{n_k}P_KFU^{n_k}\right\|\leq \left\|W^{n_k}P_K\right\|\,\|F\|.
	\end{equation*}
	This implies that $\lim_{k\rightarrow\infty}T^{n_k}_{U,W}(P_K F)=0$ in $B_0(\mathcal{H})$.	
	 Similarly, we have
	\begin{align*}
	\lim_{k\longrightarrow \infty} S^{n_k}_{U,W}(P_K F)&=\lim_{k\longrightarrow \infty} T^{n_k}_{U,W}(P_K G)
	=\lim_{k\longrightarrow \infty} S^{n_k}_{U,W}(P_KG)\\
	&=\lim_{k\longrightarrow \infty} T^{2n_k}_{U,W}(P_{E_k}G)=\lim_{k\longrightarrow \infty} S^{2n_k}_{U,W}(P_{R_k}G)=0
	\end{align*}
	in $B_0(\mathcal{H})$. Moreover, easily we obtain that 
	$$\lim_{k\rightarrow\infty} T^{n_k}_{U,W}(P_{E_k}G)=\lim_{k\rightarrow \infty} S^{n_k}_{U,W}(P_{R_k}G)=0$$
	  in $B_0(\mathcal{H})$. Therefore, setting
	$$V_k:=P_K F+2T^{n_k}_{U,W}(P_{E_k}G)+2S^{n_k}_{U,W}(P_{R_k}G),$$
	for all $k$, we have
	$$\lim_{k\longrightarrow \infty}V_k=F\quad\text{ and }\quad\lim_{k\longrightarrow \infty}C^{(n_k)}_{U,W}V_k=G.$$
	
	This completes the proof.	
\end{proof}
\begin{theorem}
	Suppose that $U,W\in B(\mathcal{H})$ such that $W$ is invertible and $U$ is unitary. Let there exist a closed subspace $K$ of $\mathcal{H}$ such that $U^n(K)\perp K$ for all $n\geq N$. Then, ${\rm (i)}\Rightarrow{\rm (ii)}$.
	\begin{itemize}
		\item [(i)] $\mathcal{P}(C^{(n)}_{U,W})$ is dense in $B_0(\mathcal{H})$, and for each $F\in B_0(\mathcal{H})$, $\lim_{n\rightarrow\infty}S^n_{U,W}(F)= 0$ in $B_0(\mathcal{H})$.
		\item [(ii)] $m(W)<1$.
	\end{itemize}
\end{theorem}
\begin{proof}
By the assumptions, we can choose  a sequence $(F_k)$ in $B_0(\mathcal{H})$ and a strictly increasing sequence $(n_k)\subseteq \mathbb{N}$ with $n_1\geq N$ such that
	\begin{equation}\label{0021}
	\|F_k-P_K\|< \frac{1}{4^k},\quad\|F_k+S^{2n_k}_{U,W}(F_k)-P_K\|<\frac{1}{4^k}
	\end{equation}
	and $C^{(n_k)}_{U,W}F_k=F_k$ for all $k\in\mathbb{N}$, where $P_K$ stands for the orthogonal projection onto $K$. Then, we have
	$$\|(F_k+S^{2n_k}_{U,W}(F_k)-P_K)\,P_K\|<\frac{1}{4^k}$$
	because $\|P_K\|=1$ and so,
		$$\|(F_k+S^{2n_k}_{U,W}(F_k))\,P_K\|>1-\frac{1}{4^k}.$$
	  This means that for each $k\in\mathbb{N}$ there exists some $0\neq x_k\in K$ such that 
	  		$$\|(F_k+S^{2n_k}_{U,W}(F_k))x_k\|>(1-\frac{1}{4^k})\,\|x_k\|.$$
	   Next, we have
	\begin{align*}
	\frac{2}{4^k}\geq 2\|F_k-P_K\|&=\|T^{n_k}_{U,W}(F_k)+S^{n_k}_{U,W}(F_k)-2P_K\|\\
	&=\|W^{n_k}F_k U^{n_k}+W^{-n_k}F_k U^{-n_k}-2P_K\|\\
	&\geq\|\left[W^{n_k}F_k U^{n_k}+W^{-n_k}F_k U^{-n_k}-2P_K\right]P_{U^{-n_k}(K)}\|\\
	&=\|\left[W^{n_k}F_k U^{n_k}+W^{-n_k}F_k U^{-n_k}\right]P_{U^{-n_k}(K)}\|\\
	&=\|\left[W^{n_k}F_k U^{n_k}+W^{-n_k}F_k U^{-n_k}\right]U^{-n_k}P_K U^{n_k}\|\\
	&=\|\left[W^{n_k}F_k U^{n_k}+W^{-n_k}F_k U^{-n_k}\right]U^{-n_k}P_K\|\\
	&=\|W^{n_k}\left[F_k+W^{-2n_k}F_kU^{-2n_k}\right]P_K\|\\
	&=\|W^{n_k}\left[F_k+S^{2n_k}_{U,W}(F_k)\right]P_K\|.
	\end{align*}
Hence, 
\begin{align*}
\frac{2}{4^k}\,\|x_k\|&\geq \|W^{n_k}\left[F_k+S^{2n_k}_{U,W}(F_k)\right]P_K\|\,\|x_k\|\\
&\geq \|W^{n_k}\left[F_k+S^{2n_k}_{U,W}(F_k)\right]P_K x_k\|\\
&=\|W^{n_k}\left[F_k+S^{2n_k}_{U,W}(F_k)\right]x_k\|\\
&\geq m(W^{n_k})\,\|\left(F_k+S^{2n_k}_{U,W}(F_k)\right)x_k\|\\
&\geq m(W^{n_k})\,(1-\frac{1}{4^k})\,\|x_k\|\\
&\geq m(W)^{n_k}\,(1-\frac{1}{4^k})\,\|x_k\|.
\end{align*}

Dividing both sides of the inequality by $\|x_k\|$ we obtain 
$$\frac{2}{4^k-1}\geq m(W)^{n_k}$$
 for all $k$. This implies that $m(W)<1$.
\end{proof}
\begin{corollary}\label{cor43}
Suppose that $U,W\in B(\mathcal{H})$ such that $W$ is invertible and $U$ is unitary. Let there exist a closed subspace $K$ of $\mathcal{H}$ such that $U^n(K)\perp K$ for all $n\geq N$.	If $m(W)> 1$, then $\{C_{U,W}^{(n)}\}$ is not chaotic on $B_0(\mathcal{H})$.
\end{corollary}
\begin{proof}	
	For each $F\in B_0(\mathcal{H})$ and $n\in\mathbb{N}$ we have 
	\begin{equation*}
	\|S_{U,W}^n F\|=\|W^{-n}FU^n\|\leq \|W^{-1}\|^n\,\|FU^n\|=\|W^{-1}\|^n\,\|F\|.
	\end{equation*}
	Now, $\|W^{-1}\|=m(W)^{-1}$, so $\|W^{-1}\|^n\rightarrow 0$ as $n\rightarrow\infty$. Thus, $\lim_{n\rightarrow\infty}S_{U,W}^n F=0$ for all $F\in B(\mathcal{H})$. Because of the previous theorem, the proof is complete.
\end{proof}
Similarly, one can obtain the following results.
\begin{theorem}
	Suppose that $U,W\in B(\mathcal{H})$ such that $W$ is invertible and $U$ is unitary. Assume that there exists a closed subspace $K$ of $\mathcal{H}$ such that $U^n(K)\perp K$ for all $n\geq N$.	We have ${\rm (i)}\Rightarrow {\rm (ii)}$.
	\begin{itemize}
		\item [(i)] $\mathcal{P}(C_{U,W}^{(n)})$ is dense in $B_0(\mathcal{H})$, and $\lim_{n\rightarrow\infty} T_{U,W}^n F=F$ for all $F\in B_0(\mathcal{H})$.
		\item [(ii)] $m(W^{-1})<1$.
	\end{itemize}
\end{theorem}

\begin{corollary}\label{cor45}
	Suppose that $U,W\in B(\mathcal{H})$ such that $W$ is invertible and $U$ is unitary. Let there exist a closed subspace $K$ of $\mathcal{H}$ such that $U^n(K)\perp K$ for all $n\geq N$. If $m(W^{-1})>1$, then $\{C_{U,W}^{(n)}\}$ is not chaotic on $B_0(\mathcal{H})$.	
\end{corollary}
\begin{corollary}
Suppose that $U,W\in B(\mathcal{H})$ such that $W$ is invertible and $U$ is unitary. Let there exist a closed subspace $K$ of $\mathcal{H}$ such that $U^n(K)\perp K$ for all $n\geq N$.	If $\{C_{U,W}^{(n)}\}$ is  chaotic on $B_0(\mathcal{H})$, then $m(W), m(W^{-1})<1$, or equivalently, $m(W)<1<\|W\|$.
\end{corollary}
\begin{proof}
Combine Corollaries \ref{cor43} and \ref{cor45}.
\end{proof}
The next result can be proved in a similar way as Theorem \ref{thm310}.
\begin{theorem}\label{thm47}
	Let $\mathcal{H}$ be a separable Hilbert space.
	We have 
	${\rm (ii)}\Rightarrow {\rm (i)}$:
	\begin{itemize}
		\item [(i)] The sequence $\{C_{U,W}^{(n)}\}$ is  chaotic on $B_0(\mathcal{H})$.
		\item [(ii)] For each $m\in\mathbb{N}$, there exists a strictly increasing  sequence $\{n_k\}\subseteq\mathbb{N}$ such that 
		\begin{equation}\label{eq111}
		\lim_{k\rightarrow\infty} \sum_{l=1}^\infty \|W^{ln_k}P_{m}\|=\lim_{k\rightarrow\infty} \sum_{l=1}^\infty \|W^{-ln_k}P_{m}\|=0,
		\end{equation}
		where the corresponding series are convergent for each $k$.
	\end{itemize}
\end{theorem}
\begin{remark}
	Our sufficient conditions for topological transitivity in the norm topology of $B_0(\mathcal{H})$ in Theorem \ref{thm32} and Theorem \ref{thm41} are also sufficient conditions for topological transitivity in the strong topology of $B(\mathcal{H})$. Indeed, since $\{e_n\}$ is an orthonormal basis for $\mathcal{H}$, it is easily seen that the set $\{P_nF:\,F\in B(\mathcal{H}),\,n\in\mathbb{N}\}$  is dense in $B(\mathcal{H})$ in the strong operator topology. Moreover, in this case the conditions \eqref{cond2}-\eqref{cond22} in Theorem \ref{thm32} can even be relaxed by considering the strong limits instead of the limit in norm and by dropping the requirement that the sequences $\{D_k\}$ and $\{G_k\}$ should belong to $B_0(\mathcal{H})$. Hence, also in the case of strong operator topology  on $B(\mathcal{H})$, the operator $W$ in Example \ref{ex34} satisfies the sufficient conditions for topological transitivity of $T_{U,W}$ and $\{C_{U,W}^{(n)}\}_n$.
\end{remark}
\begin{remark}
	Except from the implication ${\rm (i)}\Rightarrow{\rm (ii)}$ in Theorem \ref{thm32}, all our results about sufficient conditions for topological transitivity, easily generalize to the case where $B_0(\mathcal{H})$ is replaced by an arbitrary non-unital $C^*$-algebra $\mathcal{A}$, and the set of all finite rank orthogonal projections on $\mathcal{H}$ is replaced by the canonical approximate unit in $\mathcal{A}$. Indeed, if $\mathcal{A}$ is a non-unital $C^*$-algebra, then it can be isometrically embedded into a unital $C^*$-algebra $\mathcal{A}_1$ such that $\mathcal{A}$ becomes an ideal in $\mathcal{A}_1$. If $u$ and $w$ are invertible elements in $\mathcal{A}_1$ and $u$ is unitary (i.e. $uu^*=u^*u=1_{\mathcal{A}_1}$), then we can define the operator $T_{u,w}$ on $\mathcal{A}$ by
	$T_{u,w}(a):=wau$ for all $a\in\mathcal{A}$. Therefore, all our results regarding the sufficient conditions for $T_{u,w}$ to be topologically transitive or chaotic can be generalized in this setting. Especially, with the above notations, one can give the following result.
\end{remark}
\begin{theorem}
	Let $w\in\mathcal{A}_1$ be invertible and $u$ be a unitary element of $\mathcal{A}_1$.  Suppose that there exist an element $a\in\mathcal{A}^+$ and  an $N\in\mathbb{N}$ such that $au^na=0$ for all $n\geq N$.  Then, ${\rm (i)\Rightarrow {\rm (ii)}}$.
	\begin{itemize}
		\item [{\rm (i)}] $\mathcal{P}((C_{u,w}^{(n)})_n)$ is dense in $\mathcal{A}$.		
		\item [{\rm (ii)}] $m(\varphi(w))<1$, where $(\varphi,\mathcal{H})$ is the universal representation of $\mathcal{A}_1$.
	\end{itemize}
\end{theorem}
Moreover, if $\mathcal{A}$ is a unital $C^*$-algebra and $\ell_2(\mathcal{A})$ denotes the standard Hilbert module over $\mathcal{A}$, then all our results so far can be transferred directly to the case where $B_0(\mathcal{H})$ and $B(\mathcal{H})$ are replaced by $K(\ell_2(\mathcal{A}))$ and $B(\ell_2(\mathcal{A}))$, respectively. Here, $K(\ell_2(\mathcal{A}))$ and $B(\ell_2(\mathcal{A}))$ stand for the set of all compact and all bounded $\mathcal{A}$-linear operators on $\ell_2(\mathcal{A})$, respectively.
	
	
\section{Dynamics of the Adjoint Operator}
If we consider $T_{U,W}$ and $S_{U,W}$ as operators on $B_0(\mathcal{H})$, then their adjoints $T^*_{U,W}$ and $S^*_{U,W}$ are operators on $B_1(\mathcal{H})$ with the following formulas:
\begin{equation}
T^*_{U,W}(G)=UGW,\qquad S^*_{U,W}(G)=U^{-1}GW^{-1}\qquad(G\in B_1(\mathcal{H})).
\end{equation}

Indeed, if $\varphi\in B^*_0(\mathcal{H})$, then 
$$\varphi(F)=tr(GF)$$
for all $F\in B_0(\mathcal{H})$ and some $G\in B_1(\mathcal{H})$. Then,
\begin{align*}
T^*_{U,W}(\varphi)(F)&=\varphi(T_{U,W}(F))=\varphi(WFU)\\
&=tr(GWFU)=tr(UGWF)
\end{align*}
for all $F\in B_0(\mathcal{H})$.

Now, $UGW\in B_1(\mathcal{H})$ since $G\in B_1(\mathcal{H})$.


\begin{theorem}\label{thm51}
	Suppose that for every $m\in\mathbb{N}$ there exist sequences $(E_k)$ and $(R_k)$ of subspaces of $L_m$ and an increasing sequence $(n_k)\subseteq \mathbb{N}$ such that for each $k$, $L_m=E_k\oplus R_k$ and 
		\begin{equation}
	\lim_{k\rightarrow\infty}\left\|W^{n_k}\,P_{m}\right\|=\lim_{k\rightarrow\infty}\left\|W^{-n_k}\,P_{m}\right\|=0,
	\end{equation}
	\begin{equation}
	\lim_{k\rightarrow\infty}\left\|W^{2n_k}\,P_{E_k}\right\|=\lim_{k\rightarrow\infty}\left\|W^{-2n_k}\,P_{R_k}\right\|=0.
	\end{equation}
	Then, $\{C_{U,W}^{*(n)}\}$ is topologically transitive on $B_1(\mathcal{H})$.
\end{theorem}
\begin{proof}
Let $\mathcal{O}_1$ and $\mathcal{O}_2$ be non-empty open subsets of 
$B_1(\mathcal{H})$. Since the set of finite rank operators is dense in $B_1(\mathcal{H})$, we may find $G_1\in \mathcal{O}_1$ and $G_2\in \mathcal{O}_2$
such that $G_1$ and $G_2$ are finite rank operators. 
Set $\widetilde{K}:={\rm Im}G_1+{\rm Im}G_2$, then $\widetilde{K}$ is finite dimensional. Then, $P_{\widetilde{K}}G_1=G_1$ and $P_{\widetilde{K}}G_2=G_2$. Since 
$\|P_m-P_{\widetilde{K}}\|\rightarrow 0$ as $m\rightarrow\infty$ and $$\|P_mG_i-G_i\|_1=\|(P_m-P_{\widetilde{K}})G_i\|_1\leq \|P_m-P_{\widetilde{K}}\|\,\|G_i\|_1\rightarrow 0,$$
as $m\rightarrow\infty$ for $i=1,2$, we deduce that there exists an $m_0\in\mathbb{N}$ such that $P_{m_0}G_1\in\mathcal{O}_1$ and $P_{m_0}G_2\in\mathcal{O}_2$. Set $K:=L_{m_0}$, and
 choose $\{E_k\}$, $\{R_k\}$ and $\{n_k\}$ satisfying the conditions of Theorem \ref{thm51}.
Then,
\begin{align*}
\|T_{U,W}^{*^{n_k}}(P_KG_1)\|_1&=\|U^{n_k}P_K G_1 W^{n_k}\|_1\\
&=\|P_K G_1 W^{n_k}\|_1\\
&=\|G_1 W^{n_k}P_K\|_1\\
&\leq \|G_1\|_1\,\|W^{n_k}P_K\|\rightarrow 0.
\end{align*}
Similarly, 
$$\|S_{U,W}^{*^{n_k}}(P_KG_1)\|_1,\,\|T_{U,W}^{*^{2n_k}}(P_{E_k}G_2)\|_1,\,\|S_{U,W}^{*^{2n_k}}(P_{R_k}G_2)\|_1\rightarrow 0$$
as $k\rightarrow\infty$. Moreover, by similar calculations we have 
\begin{align*}
\|T_{U,W}^{*^{n_k}}(P_{E_k}G_2)\|_1&\leq \|G_2\|_1\,\|W^{n_k}P_{E_k}\|\\
&=\|G_2\|_1\,\|W^{n_k}P_KP_{E_k}\|\\
&\leq \|G_2\|_1\,\|W^{n_k}P_K\|\rightarrow 0,
\end{align*}
and likewise, $\|S_{U,W}^{*^{n_k}}(P_{R_k}G_2)\|_1\rightarrow 0$.
Set, for each $k\in\mathbb{N}$,
$$F_k:=P_KG_1+2T_{U,W}^{*^{n_k}}(P_{E_k}G_2)+2S_{U,W}^{*^{n_k}}(P_{R_k}G_2)$$
and proceed as in the proof of Theorem \ref{thm41}.
\end{proof}
\begin{theorem}
	Suppose that $U,W\in B(\mathcal{H})$ such that $W$ is invertible and $U$ is unitary. Assume that there exists a finite dimensional subspace $K$ of $\mathcal{H}$ such that $U^n(K)\perp K$ for all $n\geq N$. Then, ${\rm (i)}\Rightarrow{\rm (ii)}$.
	\begin{itemize}
		\item [(i)] $\mathcal{P}(C^{(n)^*}_{U,W})$ is dense in $B_1(\mathcal{H})$, and for each $F\in B_1(\mathcal{H})$, $\lim_{n\rightarrow\infty}{S^*}^n_{U,W}(F)= 0$ in $B(\mathcal{H})$.
		\item [(ii)] $m(W)<1$.
	\end{itemize}
\end{theorem}
\begin{proof}
Choose a sequence $\{G_k\}$ in $B_1(\mathcal{H})$ such that for each $k\in\mathbb{N}$,
$$\|G_k-P_K\|<\frac{1}{4^k},\quad \|G_k+S_{U,W}^{*^{2n_k}}G_k-P_K\|_1<\frac{1}{4^k}$$
and 
$$C_{U,W}^{*(n_k)}G_k=G_k$$
where $\{n_k\}$ is an increasing sequence in $\mathbb{N}$. We then get 
$$\|P_K(G_k+S_{U,W}^{*^{2n_k}}G_k)-P_K\|_1<\frac{1}{4^k},$$
so 
$$\|P_K(G_k+S_{U,W}^{*^{2n_k}}G_k)\|_1>(1-\frac{1}{4^k}).$$
Thus, for each $k$, there is a $D_k\in B_0(\mathcal{H})$ such that 
$$|tr(P_K(G_k+S_{U,W}^{*^{2n_k}}G_k)D_k|\geq (1-\frac{1}{4^k})\,\|D_k\|$$
and $D_k\neq 0$. Next, 
\begin{align*}
\frac{2}{4^k}&\geq 2\|G_k-P_K\|_1\\
&=\|T_{U,W}^{*^{n_k}}G_k+S_{U,W}^{*^{n_k}}G_k-2P_K\|_1\\
&=\|U^{n_k}G_k W^{n_k}+U^{-n_k}G_k W^{-n_k}-2P_K\|_1\\
&\geq \|P_{U^{n_k}(K)}(U^{n_k}G_k W^{n_k}+U^{-n_k}G_k W^{-n_k}-2P_K)\|_1\\
&=\|P_{U^{n_k}(K)}(U^{n_k}G_k W^{n_k}+U^{-n_k}G_k W^{-n_k})\|_1\\
&=\|U^{n_k}P_K U^{-n_k}(U^{n_k}G_k W^{n_k}+U^{-n_k}G_k W^{-n_k})\|_1\\
&=\|P_K U^{-n_k}(U^{n_k}G_k W^{n_k}+U^{-n_k}G_k W^{-n_k})\|_1\\
&=\|P_K (G_k+U^{-2n_k}G_k W^{-2n_k})W^{n_k}\|_1.
\end{align*}
Hence, 
\begin{align*}
\frac{2}{4^k}\,\|\widetilde{D_k}\|&\geq \|P_K (G_k+S_{U,W}^{*^{2n_k}}G_k)W^{n_k}\|_1\,\|\widetilde{D_k}\|\\
&\geq |tr(P_K (G_k+S_{U,W}^{*^{2n_k}}G_k)W^{n_k}\widetilde{D_k})|,
\end{align*}
where $\widetilde{D_k}:=W^{-n_k} D_k$. Thus, we get 
\begin{align*}
\frac{2}{4^k}\,\|\widetilde{D_k}\|&\geq |tr(P_K (G_k+S_{U,W}^{*^{2n_k}}G_k))D_k|\\
&\geq (1-\frac{1}{4^k})\,\|D_k\|=(1-\frac{1}{4^k})\,\|W^{n_k}\widetilde{D_k}\|.
\end{align*}
Next, observe that for all $x\in\mathcal{H}$ we have 
$$\|W^{n_k}\widetilde{D_k}x\|\geq (m(W))^{n_k}\,\|\widetilde{D_k} x\|.$$
Taking the supremum  over the unite sphere on the both sides of the inequality we obtain 
$$\|W^{n_k}\widetilde{D_k}\|\geq m(W)^{n_k}\,\|\widetilde{D_k}\|.$$
Hence 
$$\frac{2}{4^k}\,\|\widetilde{D_k}\|\geq (1-\frac{1}{4^k})\,m(W)^{n_k}\,\|\widetilde{D_k}\|.$$
Since $\widetilde{D_k}\neq 0$, we may divide the both sides of the inequality by $\|\widetilde{D_k}\|$, and obtain 
$$\frac{2}{4^k}\geq (1-\frac{1}{4^k})\,m(W)^{n_k}$$
which gives that 
$m(W)<1$.
\end{proof}
For each $D\in B(\mathcal{H})$, we define $\widetilde{D}\in B(B(\mathcal{H}))$ by
$\widetilde{D}(F):=FD$ for all $F\in B(\mathcal{H})$. If $\varphi\in B(\mathcal{H})'$, then we let 
$M_D\varphi\in B(\mathcal{H})'$ be given by 
$M_D\varphi(F):=\varphi(DF)$ for all $F\in B(\mathcal{H})$. If we consider now $T_{U,W}$ as an operator on $B(\mathcal{H})$, we have then that 
$$T_{U,W}^*(\varphi)=\varphi_W\circ\widetilde{U}$$
and 
$$S_{U,W}^*(\varphi)=\varphi_{W^{-1}}\circ\widetilde{U}^{-1}.$$
\begin{theorem}
	Let $U,W\in B(\mathcal{H})$ be invertible such that $U$ is unitary. Suppose that there exists a finite dimensional subspace $K$ of $\mathcal{H}$ and $N\in\mathbb{N}$ such that $U^n(K)\perp K$ for all $n\geq N$. Then, ${\rm (i)}\Rightarrow{\rm (ii)}$:
	\begin{itemize}
		\item [(i)] $\mathcal{P}\{(C^*_{U,W})^n\}$ is dense in $B(\mathcal{H})'$ and 
		$\lim_{n\rightarrow\infty} (S^*_{U,W})^n\varphi =0$ for all $\varphi\in B(\mathcal{H})'$.
		\item [(ii)] $m(W)<1$.
	\end{itemize}
\end{theorem}
\begin{proof}
there exists $\phi_K\in B(\mathcal{H})'$ such that $\phi_K(P_K)=1$ and $\|\phi_K\|=1$. Let $\widetilde{\phi_K}:=\phi_K\circ \widetilde{P_K}$.
	Choose a sequence $\{\varphi_k\}$ in the dual of $B(\mathcal{H})$ and an increasing sequence $\{n_k\}$ of positive integers s.t. 
	$$\|\varphi_k-\widetilde{\phi_K}\|<\frac{1}{4^{k+1}},\qquad\|\varphi_k+S_{U,W}^{*^{2n_k}}\varphi_k-\widetilde{\phi_K}\|<\frac{1}{4^k},$$
	and 
	$$C_{U,W}^{(n_k)^*}\varphi_k=\varphi_k$$
	for all $k\in\mathbb{N}$. 	
	Now, we obtain 
	\begin{equation*}
	\|\varphi_k\circ \widetilde{U}^n+S_{U,W}^{*^{2n_k}}\varphi_k\circ\widetilde{U}- \widetilde{\phi_K}\circ \widetilde{U}^n\|<\frac{1}{4^k},
	\end{equation*}
	which gives 
	$$\|\varphi_k\circ \widetilde{U}^n+S_{U,W}^{*^{2n_k}}\varphi_k\circ\widetilde{U}^n\|>1-\frac{1}{4^k},$$
	for all $n\in\mathbb{N}$, as $\|\widetilde{\phi_K}\circ \widetilde{U}^n\|=\|- \widetilde{\phi_K}\|=1$. Next,
	\begin{align*}
	\frac{2}{4^k}&\geq  2\|\varphi_k-\widetilde{\phi_K}\|\\
	&=\|T_{U,W}^{*^{n_k}}\varphi_k+S_{U,W}^{*^{n_k}}\varphi_k-2\widetilde{\phi_K}\|\\
	&=\|M_{W^{n_k}}\varphi_k\circ \widetilde{U}^{n_k}+M_{W^{-n_k}}\varphi_k\circ \widetilde{U}^{-n_k}-2\widetilde{\phi_K}\|.
	\end{align*}
	Next, since 
	$$\|\varphi_k+S_{U,W}^{*^{2n_k}}\varphi_k-\widetilde{\phi_K}\|<\frac{1}{4^k},$$
	we get 
	$$\|\varphi_k\circ \widetilde{P_K}+S_{U,W}^{*^{2n_k}}\varphi_k\circ \widetilde{P_K}-\widetilde{\phi_K}\circ \widetilde{P_K}\|<\frac{1}{4^k}.$$
	Since $\widetilde{\phi_K}\circ \widetilde{P_K}=\widetilde{\phi_K}$ and $\|\widetilde{\phi_K}\|=1$,  we get 
	$$\|\varphi_k\circ \widetilde{P_K}+S_{U,W}^{*^{2n_k}}\varphi_k\circ \widetilde{P_K}\|>1-\frac{1}{4^k}.$$
	Then, for each $k$, there exists an $F_k\in B(\mathcal{H})$ such that $\|F_k\|\neq 0$ and 
	$$\|(\varphi_k\circ \widetilde{P_K}+S_{U,W}^{*^{2n_k}}\varphi_k\circ \widetilde{P_K})F_k\|>\left(1-\frac{1}{4^k}\right)\|F_k\|.$$
	So, we get 
	\begin{align*}
	\frac{2}{4^k}&>\|\Big(M_{W^{n_k}}\varphi_k\circ \widetilde{U}^{n_k}+M_{W^{-n_k}}\varphi_k\circ \widetilde{U}^{-n_k}-2\widetilde{\phi_K}\Big)\circ \widetilde{P_{U^{n_k}(K)}}\|\\
	&=\|\Big(M_{W^{n_k}}\varphi_k\circ \widetilde{U}^{n_k}+M_{W^{-n_k}}\varphi_k\circ \widetilde{U}^{-n_k}\Big)\circ \widetilde{P_{U^{n_k}(K)}}\|\\
	&=\|\Big(M_{W^{n_k}}\varphi_k\circ \widetilde{U}^{n_k}+M_{W^{-n_k}}\varphi_k\circ \widetilde{U}^{-n_k}\Big)\circ\widehat{U}^{-n_k} \widetilde{P_K}\circ\widehat{U}^{n_k}\|\\
	&=\|\Big(M_{W^{n_k}}\varphi_k\circ \widetilde{U}^{n_k}+M_{W^{-n_k}}\varphi_k\circ \widetilde{U}^{-n_k}\Big)\circ\widehat{U}^{-n_k} \widetilde{P_K}\|\\
	&=\|\Big(M_{W^{n_k}}\varphi_k+M_{W^{-n_k}}\varphi_k\circ \widetilde{U}^{-2n_k}\Big)\circ\widetilde{P_K}\|\\
	&=\|M_{W^{n_k}}\Big(\varphi_k+M_{W^{-2n_k}}\varphi_k\circ \widetilde{U}^{-2n_k}\Big)\circ\widetilde{P_K}\|\\
	&=\|M_{W^{n_k}}\Big([\varphi_k+M_{W^{-2n_k}}\varphi_k\circ \widetilde{U}^{-2n_k}]\circ P_K\Big)\|\\
	&=\|M_{W^{n_k}}\Big([\varphi_k+S_{U,W}^{*^{2n_k}}\varphi_k]\circ \widetilde{P_K}\Big)\|.
	\end{align*}
	Hence, 
	\begin{align*}
	\frac{2}{4^k}\,\|W^{-n_k}F_k\|&\geq \|M_{W^{n_k}}\Big([\varphi_k+S_{U,W}^{*^{2n_k}}\varphi_k]\circ \widetilde{P_K}\Big)\|\,\|W^{-n_k}F_k\|\\
	&\geq \|M_{W^{n_k}}\Big([\varphi_k+S_{U,W}^{*^{2n_k}}\varphi_k]\circ \widetilde{P_K}\Big)\Big(W^{-n_k}F_k\Big)\|\\
	&=\|(\varphi_k+S_{U,W}^{*^{2n_k}}\varphi_k)\circ \widetilde{P_K})\Big(F_k\Big)\|\\
	&\geq (1-\frac{1}{4^k})\|F_k\|\\
	&=(1-\frac{1}{4^k})\|W^{n_k}(W^{-n_k}F_k)\|.	
	\end{align*}
	Now, observe that for all $x\in\mathcal{H}$ we have 
	$$\|W^{n_k}(W^{-n_k}F_k)x\|\geq (m(W))^{n_k}\|(W^{-n_k}F_k)x\|.$$
	Taking supremum over the unit ball in $\mathcal{H}$ on the both sides of the inequality, we obtain 
	$$\|F_k\|\geq (m(W))^{n_k} \|(W^{-n_k}F_k)\|.$$
	It follows that $m(W)<1$.
\end{proof}
Consider $B(\mathcal{H})$ as a topological vector space equipped with the strong topology. Let $B(\mathcal{H})'$ be equipped with $w^*$-topology.
\begin{theorem}\label{thm54}
We have ${\rm (ii)}\Rightarrow{\rm (i)}$: 
\begin{itemize}
	\item [(i)] $(C_{U,W}^{(n)*})$ is topologically transitive in $B(\mathcal{H})'$.
	\item [(ii)] For every $m\in\mathbb{N}$ there exist sequences $(E_k)$ and $(R_k)$ of subspaces of $L_m$ and an increasing sequence $(n_k)\subseteq \mathbb{N}$ such that for each $k$, $L_m=E_k\oplus R_k$ and 
	\begin{equation}
	\lim_{k\rightarrow\infty}\left\|P_{m}\,W^{n_k}\right\|=\lim_{k\rightarrow\infty}\left\|P_{m}\,W^{-n_k}\right\|=0,
	\end{equation}
	\begin{equation}
	\lim_{k\rightarrow\infty}\left\|P_{E_k}\,W^{2n_k}\right\|=\lim_{k\rightarrow\infty}\left\|P_{R_k}\,W^{-2n_k}\right\|=0.
	\end{equation}
\end{itemize}
\end{theorem}
\begin{proof}
Let $\mathcal{O}_1$ and $\mathcal{O}_2$ be two non-empty open subsets of $B(\mathcal{H})'$ in the $w^*$-topology, choose some $\varphi_1\in\mathcal{O}_1$ and $\varphi_2\in\mathcal{O}_2$.  Then, for each $F\in B(\mathcal{H})$ we have $P_n F\rightarrow F$ strongly as $n\rightarrow\infty$. Hence, $\varphi_1(P_nF)\rightarrow \varphi_1(F)$ and $\varphi_2(P_n F)\rightarrow \varphi_2(F)$ for all $F\in B(\mathcal{H})$. It follows that $M_{P_n}\varphi_1\rightarrow \varphi_1$ and $M_{P_n}\varphi_2\rightarrow \varphi_2$ in the $w^*$-topology a $n\rightarrow\infty$. Therefore, there exists an $n_0\in\mathbb{N}$ such that $M_{P_{n_0}}\varphi_1\in\mathcal{O}_1$ and $M_{P_{n_0}}\varphi_2\in\mathcal{O}_2$. Set $K:=L_{n_0}$.
For every $F\in B(\mathcal{H})$ we have 
$$|T_{U,W}^{*^{2n_k}}M_{P_{E_k}}\varphi_2)(F)|=|\varphi_2(P_{E_k}W^{2n_k} FU^{2n_k})|.$$
Since 
$$\|P_{E_k}W^{2n_k} FU^{2n_k}\|\leq \|P_{E_k}W^{2n_k}\|\,\|F\|,$$
we have 
$$\varphi_2(P_{E_k}W^{2n_k} FU^{2n_k})\rightarrow 0,$$
as $k\rightarrow\infty$, because $\varphi_2$ is continuous in the strong topology. As this holds for all $F\in B(\mathcal{H})$, we deduce that 
$T_{U,W}^{*^{2n_k}}(M_{P_{E_k}}\varphi_2)\rightarrow 0$ in the $w^*$-topology as $k\rightarrow\infty$.

 Similarly, $S_{U,W}^{*^{2n_k}}(M_{P_{R_k}}\varphi_2)\rightarrow 0$ in the $w^*$-topology as $k\rightarrow\infty$. 
 
 Moreover, 
$$T_{U,W}^{*^{n_k}}(M_{P_{E_k}}\varphi_1)\rightarrow 0,$$
$$S_{U,W}^{*^{n_k}}(M_{P_{R_k}}\varphi_1)\rightarrow 0,$$
as $k\rightarrow\infty$, in the $w^*$-topology, since 
$$\|P_{E_k}W^{n_k}\|=\|P_{E_k}P_K W^{n_k}\|\leq \|P_K W^{n_k}\|$$
and 
$$\|P_{R_k}W^{n_k}\|=\|P_{R_k}P_K W^{n_k}\|\leq \|P_K W^{n_k}\|.$$
It is not hard to see that 
also 
$$T_{U,W}^{*^{n_k}}(M_{P_{K}}\varphi_2)\rightarrow 0,$$
$$S_{U,W}^{*^{n_k}}(M_{P_{K}}\varphi_2)\rightarrow 0,$$
as $k\rightarrow\infty$, in the $w^*$-topology. 

Then, for each $k\in\mathbb{N}$ put 
$$\psi_k:=M_{P_K}\varphi_1+2T_{U,W}^{*^{n_k}}(M_{P_{E_k}}\varphi_2)+2S_{U,W}^{*^{n_k}}(M_{P_{R_k}}\varphi_2),$$
and proceed as in the proof of Theorem \ref{thm41} and Theorem \ref{thm51}.
\end{proof}
\begin{theorem}
	We have 
	${\rm (i)}\Rightarrow {\rm (ii)}$:
	\begin{itemize}
		\item [(i)] $P(T_{U,W}^{*^n})$ is dense in $B(\mathcal{H})'$.
		\item [(ii)] $m(W)<1$.
	\end{itemize}
\end{theorem}
\begin{proof}
	Let $K$ be a finite dimensional subspace of $\mathcal{H}$ s.t. $U^n(K)\perp K$ for all $n\geq N$. Let $\phi_K\in B(\mathcal{H})'$ be such that $\phi_K(P_K)=1$ and $\|\phi_K\|=1$. Set $\widetilde{\phi_K}:=\phi_K\circ\widetilde{P_K}$. Then, $\widetilde{\phi_K}(P_K)=1$ and 
	$\|\widetilde{\phi_K}\|=1$. For each $k\in\mathbb{N}$, there exists $\varphi_k\in \mathcal{P}(T_{U,W}^{*^n})$ such that 
	$$\frac{1}{k^2}\geq \|\varphi_k -\widetilde{\phi_K}\|.$$
	Hence 
	$$\frac{1}{k^2}\geq \|(\varphi_k-\widetilde{\phi_K})\circ \widetilde{P_K}\|=\|(\varphi_k\circ \widetilde{P_K}-\widetilde{\phi_K})\|,$$
	which gives 
	$$\|\varphi_k\circ\widetilde{P_K}\|>1-\frac{1}{k^2}.$$
	Thus, for each $k\in\mathbb{N}$, there exists an $F_k\in B(\mathcal{H})$ s.t. $F_k\neq 0$ and 
	$$|\varphi_k(F_k P_K)|\geq (1-\frac{1}{k^2})\,\|F_k\|.$$
	We may assume that there exists a strictly increasing sequence of positive integers 
	$N<n_1<n_2<\cdots$ such that
	$T_{U,W}^{*^{n_k}}\varphi_k=\varphi_k$. We get 
	\begin{align*}
	\frac{1}{k^2}&\geq \|\varphi_k-\widetilde{\phi_K}\|\\
	&=\|T_{U,W}^{*^{n_k}}\varphi_k-\widetilde{\phi_K}\|\\
	&=\|(M_{W^{n_k}}\varphi_k)\circ \widetilde{U}^{n_k}-\widetilde{\phi_K}\|\\
	&\geq \|\Big((M_{W^{n_k}}\varphi_k)\circ \widetilde{U}^{n_k}-\widetilde{\phi_K}\Big)\circ \widetilde{P_{U^{n_k}(K)}}\|\\
	&=\|\Big((M_{W^{n_k}}\varphi_k)\circ \widetilde{U}^{n_k}\Big)\Big(\widetilde{U}^{-n_k}\circ\widetilde{P_K}\circ\widetilde{U}^{n_k}\Big)\|\\
	&=\|(M_{W^{n_k}}\varphi_k)\circ\widetilde{P_K}\|=\|(M_{W^{n_k}}(\varphi_k\circ\widetilde{P_K})\|.
	\end{align*}
	Hence, 
	\begin{align*}
	\frac{1}{k^2}\,\|W^{-n_k}F_k\|&\geq \|(M_{W^{n_k}}(\varphi_k\circ\widetilde{P_K})\|\,\|W^{-n_k}F_k\|\\
	&\geq \|(M_{W^{n_k}}(\varphi_k\circ\widetilde{P_K})(W^{-n_k}F_k)\|\\
	&=|\varphi_k(F_k P_K)|\\
	&\geq (1-\frac{1}{k^2})\,\|F_k\|\\
	&\geq (1-\frac{1}{k^2})\,(m(W))^{n_k}\,\|W^{-n_k}F_k\|.
	\end{align*}
	Dividing on the both sides of the inequality by 
	$\|W^{-n_k}F_k\|$ and letting $k\rightarrow\infty$, we obtain the implication.
\end{proof}
\begin{theorem}
	We have 
	${\rm (i)}\Rightarrow {\rm (ii)}$:
	\begin{itemize}
		\item [(i)] $P(S_{U,W}^{*^n})$ is dense in $B(\mathcal{H})'$.
		\item [(ii)] $m(W^{-1})=\|W\|^{-1}<1$, that is $\|W\|>1$.
	\end{itemize}
\end{theorem}
\begin{proof}	
	Similar to the previous theorem.
\end{proof}
\begin{theorem}\label{thm57}
	Let $B(\mathcal{H})$ be equipped with the strong topology, and $B(\mathcal{H})'$ be equipped with the $w^*$-topology, where $B(\mathcal{H})'$ is the dual of $B(\mathcal{H})$. Then we have ${\rm (ii)}\Rightarrow {\rm (i)}$:
	\begin{itemize}
		\item [(i)] $\{T_{U,W}^{*^n}\}$ and $\{S_{U,W}^{*^n}\}$ are topologically transitive on $B(\mathcal{H})'$.
		\item [(ii)] for every $n\in\mathbb{N}$ there exist an increasing sequence $\{n_k\}\subseteq \mathbb{N}$ and sequences of operators $\{G_k\}$ and $\{D_k\}$ in $B(\mathcal{H})$ such that same as theorem 3.2 in the draft with 
		$$\lim_{k\rightarrow \infty}\|G_kW^{n_k}\|=\lim_{k\rightarrow \infty}\|D_kW^{-n_k}\|=0,$$
		and 
		$${\rm s}\!-\!\!\lim_{k\rightarrow \infty}G_k={\rm s}\!-\!\!\lim_{k\rightarrow \infty}D_k=P_n,$$
		where ${\rm s}\!-\!\!\lim$ denotes the limit in the strong operator topology.
	\end{itemize}
\end{theorem}
\begin{proof}
	Let $\mathcal{O}_1$ and $\mathcal{O}_2$ be two non-empty open subsets of $B(\mathcal{H})'$ in the $w^*$-topology. Then, $P_nF\rightarrow F$ in strong topology for each $F\in B(\mathcal{H})$. If $\varphi_1\in \mathcal{O}_1$ and $\varphi_2\in \mathcal{O}_2$, then $M_{P_n}\varphi_1\rightarrow\varphi_1$ and $M_{P_n}\varphi_2\rightarrow\varphi_2$ as $n\rightarrow\infty$ in $w^*$-topology. Therefore, there exists an $n\in\mathbb{N}$ such that 
	$M_{P_n}\varphi_1\in\mathcal{O}_1$ and $M_{P_n}\varphi_2\in\mathcal{O}_2$. Set $K:=L_n$.
	 We get for all $D,L\in B(\mathcal{H})$,
	$$T_{U,W}^{*^{n_k}}(M_D\varphi_1)(L)=\varphi_1(DW^{n_k}LU^{n_k}).$$
	Hence, for each $L\in B(\mathcal{H})$ we have
	\begin{align*}
	|T_{U,W}^{*^{n_k}}(M_{P_nG_k}\varphi_1)(L)|&=|\varphi_1(P_nG_kW^{n_k}LU^{n_k})|.
	\end{align*} 
	 Thus, $T_{U,W}^{*^{n_k}}(M_{P_nG_k}\varphi_1)\rightarrow 0$ as $k\rightarrow\infty$, in $w^*$-topology. Similarly, $$S_{U,W}^{*^{n_k}}(M_{P_nD_k}\varphi_1)\rightarrow 0,\quad T_{U,W}^{*^{n_k}}(M_{P_nG_k}\varphi_2)\rightarrow 0,\quad S_{U,W}^{*^{n_k}}(M_{P_nD_k}\varphi_2)\rightarrow 0,$$
	  in $w^*$-topology. Set 
	$$\psi_k:=M_{P_nG_k}\varphi_1+S_{U,W}^{*^{n_k}}(M_{P_nD_k}\varphi_2)$$
	and 
	$$\eta_k:=M_{P_nD_k}\varphi_1+T_{U,W}^{*^{n_k}}(M_{P_nG_k}\varphi_2).$$
	Then,
	$\psi_k\rightarrow M_{P_n}\varphi_1$ and $T_{U,W}^{*^{n_k}}(\psi_k)\rightarrow M_{P_n}\varphi_2$. Also,
	$\eta_k\rightarrow M_{P_n}\varphi_1$ and $S_{U,W}^{*^{n_k}}(\eta_k)\rightarrow M_{P_n}\varphi_2$, in $w^*$-topology. This completes the proof.
\end{proof}

\vspace{.1in}

\begin{thebibliography}{90}

\bibitem{bmbook} F. Bayart and \'E. Matheron, Dynamics of Linear Operators, Cambridge Tracts in Math. \textbf{179},
Cambridge University Press, Cambridge, 2009.

\bibitem{bg99} L. Bernal-Gonz\'alez, {\it On hypercyclic operators on Banach
	spaces}, Proc. Amer. Math. Soc. \textbf{127} (1999) 1003-1010.


\bibitem{chache} S-J. Chang and C-C. Chen, {\it Topological mixing for cosine operator functions generated by shifts}, Topol. Appl. {\bf 160} (2013) 382-386.

\bibitem{chen11} C-C. Chen, {\it Chaotic weighted translations on groups}, Arch. Math. \textbf{97} (2011) 61-68.

\bibitem{chen141} C-C. Chen, {\it Chaos for cosine operator functions generated by shifts},
Int. J. Bifurcat. Chaos \textbf{24} (2014) Article ID 1450108, 7 pages.

\bibitem{chen15} C-C. Chen, {\it Topological transitivity for cosine operator functions on groups}, Topol. Appl.
\textbf{191} (2015) 48-57.

\bibitem{ccot} C-C. Chen, K-Y. Chen, S. \"Oztop and S.M. Tabatabaie, {\it Chaotic translations on weighted Orlicz spaces}, Ann. Polon. Math. \textbf{122} (2019) 129-142.

\bibitem{cc11} C-C. Chen and C-H. Chu, {\it Hypercyclic weighted translations on groups}, Proc. Amer. Math.
Soc. \textbf{139} (2011) 2839-2846.

\bibitem{cd18} C-C. Chen and W-S. Du, {\it Some characterizations of disjoint topological transitivity on Orlicz spaces},
J. Inequalities and Applications \textbf{2018} 2018:88.

\bibitem{chta2} C-C. Chen and S.M. Tabatabaie, {\it Chaotic operators on hypergroups}, Oper. Matrices, \textbf{12}(1) (2018) 143-156.

\bibitem{chta3} C-C. Chen and S.M. Tabatabaie, {\it Chaotic and hypercyclic operators on solid Banach function spaces}, Probl. Anal. Issues Anal., to appear.


\bibitem{conw} J.B. Conway, A Course in Functional Analysis, Springer-Verlag, New York, 1985.



\bibitem{ge00} K-G. Grosse-Erdmann, {\it Hypercyclic and chaotic weighted shifts}, Studia Math. \textbf{139} (2000) 47-68.

\bibitem{gpbook} K-G. Grosse-Erdmann and A. Peris, Linear Chaos, Universitext, Springer, 2011.

\bibitem{kalmes10} T. Kalmes, {\it Hypercyclicity and mixing for cosine operator functions generated by second order
	partial differential operators}, J. Math. Anal. Appl. \textbf{365} (2010) 363-375.

\bibitem{kostic} M. Kosti\'c, {\it Hypercyclic and chaotic integrated $C$-cosine functions}, Filomat \textbf{26} (2012) 1-44.

\bibitem{sha} S-Y. Shaw, {\it Growth order and stability of semigroups and cosine operator functions}, J. Math. Anal. Appl. {\bf 357} (2009) 340-348.
\end{thebibliography}
\end{document}